\documentclass{amsart}
\usepackage[english]{babel}
\usepackage{amsmath}
\usepackage{mathscinet}
\usepackage{amssymb}
\usepackage[table,xcdraw]{xcolor}
\usepackage{amsthm}
\usepackage{amsfonts}
\usepackage{url}
\usepackage[margin=1.4in]{geometry}
\usepackage{dsfont}
\usepackage{ mathrsfs }
\usepackage{graphicx, subfigure}

\usepackage[new]{old-arrows}

\newcommand{\R}{\mathbb{R}}
\newcommand{\Z}{\mathbb{Z}}
\newcommand{\N}{\mathbb{N}}

\newcommand{\C}{\mathbb{C}}

\newcommand{\p}[1]{\left(#1\right)}

\newcommand{\Span}[1]{\mathrm{Span}\left\lbrace #1\right\rbrace}

\theoremstyle{plain}
\newtheorem{theorem}{Theorem}[section]

\newtheorem{proposition}[theorem]{Proposition}
\newtheorem{lemma}[theorem]{Lemma}

\newtheorem{introtheorem}{Theorem}

\theoremstyle{definition}
\newtheorem{rem}[theorem]{Remark}

\newcommand{\co}{\leq_{\mathrm{c.o.}}}
\newcommand{\charfunc}[1]{{\mathds{1}}_{#1}}
\newcommand{\Mat}{{\mathrm{M}}}

\DeclareMathOperator{\Jac}{Jac}
\DeclareMathOperator{\End}{End}

\newcommand{\calH}{{\mathcal{H}}}
\newcommand{\calB}{{\mathcal{B}}}
\newcommand{\calM}{{\mathcal{M}}}

\title[On Uniform Admissibility]{On Uniform Admissibility of Unitary and Smooth Representations}

\author{Uriya A. First}
\email{uriya.first@gmail.com}
\address{Department of Mathematics\\University of Haifa\\ Mount Carmel, Haifa, Israel 31905}
\author{Thomas R\"{u}d}
\email{thomas@math.ubc.ca}
\address{Department of Mathematics\\University of British Columbia\\ BC, Canada V6T 1Z2}
\begin{document}

\maketitle

%\section{Overview}

Let $G$ be a  reductive algebraic group over a non-Archimedean local field $F$,
or more precisely, its group of $F$-points.
The study of complex representations of $G$ often
divides into investigating  two classes of representations:  
unitary representations, i.e.\ representations
on Hilbert spaces via unitary operators, and smooth representations, i.e.\
representation on an abstract $\C$-vector space such that the stabilizers of all
vectors are   open. 
Historically, the unitary representations served as the original object of interest,
and the smooth representations were introduced in  as a tool to study them.
Interplay between the  two types of representations is abundant 
and is now part of the standard theory.

Building on the work of Harish-Chandra \cite{HarishChandra},
Bernstein \cite{BernsteinPaperOriginal} proved that 
the irreducible unitary representations of $G$
are uniformly admissible.
That is to say, for any compact open subgroup  $K\leq G$, there exists an integer
$n=n(K,G)$ such that the dimension of the $K$-fixed vectors in any irreducible
unitary 
representation of $G$ is at most $n$.
Here, ``irreducible'' should be interpreted in the   category of unitary
representations,
namely, the irreducible unitary  representations 
are those representations admitting no proper \emph{closed}  subrepresentations other than~$0$.
Bernstein later showed   that the irreducible smooth
representations of $G$ are  uniformaly admissible as well;
here, ``irreducible'' means irreducible as a complex $G$-module.
A standard source for the latter result  is Section~II.2.2 
in Bernstein's lecture notes
\cite{BernsteinNotes}.
Both  results are fundamental in the theory of reductive $p$-adic Lie groups.

It was noted by  
Bernstein and Zelevinsky in \cite{BernZel76}
%\cite[pp.~4, 43]{BernZel76},
that by combining \cite[\S2.5, Prp.~2.12]{BernZel76}
with a modification of results by Godement
\cite[Lem.~4]{Godement52}, one can formally
derive uniform admissiblity of the irreducible
unitary representations of $G$ from uniform admissiblity of 
its irreducible smooth representations. 
The purpose of this note is to extend this to arbitrary $\ell$-groups, 
provide a converse, and give self-contained proofs in the process.
We show:

%and Bernstein--Zelevinski
%\cite[\S2.5, Prp.~2.12]{BernZel76}
%
% uniform admissiblity of irreducible smooth representations implies
%uniform admissibility 
%
%
%The purpose of this work is to show that they formally imply each other,
%eliminating the need of giving two separate proofs.
%We show: 

%A fundamental result in the theory of $p$-adic Lie groups
%states that the irreducible unitary representations of $G$
%are uniformly admissible, and likewise for the irreducible smooth
%representations.
%That is to say, for any compact open subgroup  $K\leq G$, there exists an integer
%$n=n(K,G)$ such that the dimension of the $K$-fixed vectors in any irreducible
%unitary, resp.\ smooth,
%representation of $G$ is at most $n$.
%Here, ``irreducible'' should be interpreted in the relevant category,
%namely, irreducible unitary  representations
%are those representations admitting no nontrivial \emph{closed}  subrepresentation,
%and irreducible smooth representations are those which are irreducible
%as complex $ G$-modules.
%
%The unitary case and the smooth case  require two different proofs, both 
%due to Bernstein.
%Uniform admissibility for irreducible representations
%was shown in \cite{BernsteinPaperOriginal}, relying on Harish-Chandra's work \cite{HarishChandra},
%whereas the standard source for uniform admissibility of irreducible smooth
%representations is \cite{BernNotes}.
%
%In this work, we show that in general it is sufficient 
%to   establish either the unitary or the smooth case, and 
%deduce the other formally. Specifically, we show:

\begin{introtheorem}\label{TH:A}
	Let $G$ be a  totally disconnected locally compact group such
	that $|G/K|<{2^{\aleph_0}}$ for some compact open subgroup $K\leq G$. Then
	the following conditions are equivalent:
	\begin{enumerate}
		\item[(a)] The irreducible unitary representations of $G$
		are uniformly admissible.
		\item[(b)] The irreducible smooth representations of $G$
		are uniformly admissible.
	\end{enumerate}
	In fact, (b) implies (a) without the assumption $|G/K|<{2^{\aleph_0}}$.
\end{introtheorem}

The assumption   $|G/K|<{2^{\aleph_0}}$ is necessary in general; see Remark~\ref{RM:counterexamples}.
The proof makes use of polynomial identities, a technique which
goes back at least as far as Kaplansky \cite{Kaplansky49}.

Theorem~\ref{TH:A} was applied in \cite[\S4.2]{Solleveld17} to show that
irreducible unitary representations of  \emph{quasi-reductive} groups
over a non-Archimedean local field are uniformly admissible.

\medskip

In addition to Theorem~\ref{TH:A}, we   show, separately, that uniform admissibility of the  
irreducible smooth representations is  invariant  under passing to a subgroup or
an overgroup of finite index.

\begin{introtheorem}\label{TH:B}
	Let $G$ be a  totally disconnected locally compact group
	and let $H\leq G$ be an open subgroup of finite index. Then
	the following conditions are equivalent:
	\begin{enumerate}
		\item[(a)] The irreducible smooth representations of $G$
		are uniformly admissible.
		\item[(b)] The irreducible smooth representations of $H$
		are uniformly admissible.
	\end{enumerate}
	This statement also holds upon replacing ``uniformly admissible''
	with ``admissible''. 
\end{introtheorem} 

We do not know whether Theorem~\ref{TH:B} holds under the milder
assumption that $H$ is a group such that there exists
a continuous group homomorphism $H\to G$ with finite kernel and cokernel.

%uniform admissibility of unitary
%representations is equivalent to uniform admissibility of smooth representation
%for any   totally disconnected locally compact group $G$.
%
%
%In the theory of $p$-adic Lie groups, i.e.\ algebraic groups
%over local non-archimedean fields, in addition to unitary
%representions, it is also fruitful to consider
%smooth representations. 

\medskip

	The paper is organized in three sections, the first of which   is preliminary, whereas
	the second and the third are concerned with proving Theorems~\ref{TH:A} and~\ref{TH:B},
	respectively.
	
\medskip

	We thank Maarten Solleveld for  encouraging us to publish this note.

\section{Preliminaries}

We  begin with recalling some necessary facts and setting notation for the sequel.
Throughout, all vector spaces are over $\C$, and an algebra
means a unital associative $\C$-algebra.
%
 %In this section, we let 

%\medskip

\subsection{{$\boldsymbol\ell$}-Groups} 
 
 Throughout this paper, $G$ denotes a  locally compact totally disconnected group, or an
 $\ell$-group for short,
 and $\mu$
 is a fixed left-invariant  measure on $G$. 
 The modular character of $G$ is denoted
 $\delta:G\to \R_{>0}$; it 
 is determined by the identity $\mu(Sg)=\mu(S)\delta(g)$, holding for any compact
 $S\subseteq G$ and $g\in G$.

 %If $K$ is a compact open subgroup of $G$, we write $K\co G$.
 We write $K\co G$ to denote that $K$ is a compact open subgroup of $G$. 
 It is well known that such subgroups
 form a basis of neighborhoods to the identity element $1_G$; see, for instance, \cite[\S2.3]{MontZipp55_groups}.

\medskip

	As usual,
	a unitary representation of $G$
	consists of a complex Hilbert space $V$ endowed with a continuous
	$G$-action $G\times V\to V$ such that $\langle gu,gv\rangle = \langle u,v\rangle$
	for all $u,v\in V$, $g\in G$. We say that $V$ is topologically irreducible,
	or just irreducible, if $V$ has no proper nonzero closed
	$G$-submodules.
	
	A smooth representation of $G$
	is an abstract complex vector space $V$ together with a
	$G$-action $G\times V\to V$ such that the $G$-stabilizer of every
	$v\in V$ is open, or equivalently, $V=\bigcup_{K\co G}V^K$,
	where $V^K:=\{v\in V\,:\,\text{$gv=v$ for all $g\in K$}\}$.
	We say that $V$ is algebraically irreducible, or just irreducible,
	if $V$ is contains no proper nonzero $G$-modules.
	
%	Every unitary representation $V$ contains a maximal smooth
%	subrepresentation, $V_{\mathrm{sm}}:=\bigcup_{K\co G} V^K$,
%	which is then dense in $V$.

\subsection{Unital {$\boldsymbol*$}-Algebras}

	By a (unital) $*$-algebra, we mean a  unital associative  $\C$-algebra
	$A$ together with an involution $*:A\to A$, i.e., 
	an additive map satisfying
	$a^{**}=a$, $(ab)^*=b^*a^*$
	and   
	$(\alpha a)^*=\overline{\alpha} a^*$ for all $a\in A$, $\alpha\in \C$.
	
	A unitary representation of a $*$-algebra $A$
	is a Hilbert space $V$ endowed with an $A$-module
	structure such that $\langle au,v\rangle=\langle u,a^*v\rangle$
	for all $a\in A$, $u,v\in V$. A unitary representation 
	is topologically irreducible, or just irreducible, if
	it has no proper nonzero closed $A$-submodules.
	%We write $\End_A(V)$ for the \emph{continuous} $A$-endomorphisms
	%of a unitary representation $V$.
	
\medskip

	We shall need the following versions of Schur's lemma.
	
	\begin{proposition}\label{LM:Schur-smooth}
		Let $A$ be a $*$-algebra with $\dim_{\C} A<{2^{\aleph_0}}$
		and let $M$ be a simple left $A$-module.
		Then $\End_A(M)=\C$.
	\end{proposition}	
	
	\begin{proof}
		Fix some $0\neq m\in M$.
		Then $Am=M$ and hence any element
		$f\in \End_A(M)$ is uniquely determined by $f(m)$.
		As a result, $\dim_{\C}\End_A(M)\leq \dim_{\C} M\leq \dim_{\C}A<{2^{\aleph_0}}$.
		Thus, $\End_A(M)$ is a division algebra over $\C$ of $\C$-dimension $<|\C|$.
		It is well known that such algebras are algebraic over $\C$,
		and hence coincide with $\C$. We conclude that $\End_A(M)=\C$.
		%The proposition follows.
	\end{proof}	
	
	\begin{proposition}\label{LM:Schur-unitary}
		Let $A$ be a $*$-algebra  
		and let $V$ be an irreducible unitary
		representation.
		Then all continuous $A$-endomorphisms of $V$
		are    scalar.
	\end{proposition}
	
	\begin{proof}
		See, for instance, \cite[Thm.~9.6.1]{Palmer_algebra}
	\end{proof}
	
\subsection{The Relative Hecke Algebra}
\label{subsec:rel-Hecke-alg}
	
	The most important example of a $*$-algebra that we shall consider
	is   the relative Hecke algebra $\calH(G,K)$
	associated to a compact open subgroup $K\co G$.
	Recall that the underlying vector space
	of $\calH(G,K)$ consists of
	the bi-$K$-invariant compactly supported functions
	$f:G\to \C$
	and that its multiplication is given by convolution
	\[(f\star g)(y):= \int_{x\in G}f(x)g(x^{-1}y)\ \text{d}\mu  \ .\] 
	%%%%% Uriya: verified the associativity
	The unity of $\mathcal{H}(G,K)$
	is $e_K:=\mu(K)^{-1}\mathds{1}_K$, where $\mathds{1}_K:G\to \C$
	denotes the characteristic function of $K$.
	We make $\calH(G,K)$ into a $*$-algebra
	by setting
	\[
	f^*(g)=\delta(g)\overline{f(g^{-1})}
	\]
	%%%% Uriya: verified that this is an involution
	for all $f\in\calH(G,K)$ and $g\in G$ (recall that $\delta$ is the modular
	character of $G$).

\medskip

	Given a unitary or smooth representation $V$
	of $G$, the space $V^K$ carries
	a left $\calH(G,K)$-module structure
	given by
	\[
	f\cdot v:=\int_{g\in G}f(g)\cdot gv\,\mathrm{d}\mu 
	\qquad\forall\, f\in \calH(G,K),\,v\in V^K	
	\]
	%%%% Uriya: verified that this indeed gives a left-module structure (associativity)
	Note that in the smooth case, if we write
	$f=\sum_{i} \alpha_i\mu(K)^{-1}\charfunc{g_iK}$, then the integral is 
	just the finite sum $\sum_i\alpha_ig_iv$.
	
	When $V$ is a unitary representation, 
	we further have
	$\langle fu,v\rangle = \langle u,f^*v\rangle$
	for all $u,v\in V$ and $f\in\calH(G,K)$, and so $V^K$
	becomes  a unitary representation of $\mathcal{H}(G,K)$.
	We remark that the equality
	$\langle fu,v\rangle = \langle u,f^*v\rangle$ follows
	easily from the definitions and the identity
	$\int_{x\in G}f(x^{-1})\delta(x)\mathrm{d}\mu =\int_{x\in G}f(x)\mathrm{d}\mu$,
	which holds whenever the integrals exist.
	%%%% Uriya: I verified that.	
	
%
%	Given $K\co G$, we let $\mathcal{H}(G,K)$ denote
%	the relative Hecke algebra of $(G,K)$;
%	its underlying vector space is the space of 
%	bi-$K$-invariant compactly supported functions $G\to \C$,
% 	and its multiplication is given by convolution
%	 \[(f\star g)(y):= \int_{x\in G}f(x)g(x^{-1}y)\ \text{d}\mu  \ .\] 
%	The unity of $\mathcal{H}(G,K)$
%	is $e_K:=\mu(K)^{-1}\mathds{1}_K$, where $\mathds{1}_K:G\to K$
%	denotes the characteristic function of $K$.
%
%\medskip
%	
%	By an involution on a (unital associative) $\C$-algebra
%	$A$, we mean a map $*:A\to A$ which is additive, involutary reverses
%	the order of multiplication and satisfies
%	$(\alpha a)^*=\overline{\alpha} a^*$ for all $a\in A$, $\alpha\in \C$.
%	We then call $(A,*)$ a $*$-algebra.
%	The algebra $\mathcal{H}(G,K)$ carries
%	a canonical involution $*$ sending $f$
%	to $f^*$ defined by
%	\[
%	f^*(g)=\overline{f(g^{-1})}\ .
%	\]
%	
%	More generally, a (unital associative) $\C$-algebra $A$ with a map $*:A\to A$ satisfying
%	the previous axioms is called a \emph{$*$-algebra}.

\medskip

	We shall need the following two easy propositions.

%\begin{proposition} \label{GtoHecke}If $V$ is an algebraically irreducible smooth representation of $G$, then for all compact open $K\leq G$ , either $V^K=0$ or $V^K$ is a simple $\mathcal{H}_K(G)-$module.
%\end{proposition}
%\begin{proof}
%Let $V$ be an algebraically irreducible smooth representation of $G$. Let $0\neq v\in V^K$, for some  $K\leq G$ compact open. Let $g\in G$ and let $f =  \mu(K)^{-1}e_K\star \mathds{1}_{gK}(x)\in \mathcal{H}_K(G)$. A short computation gives us that $fv = e_K(gv)\in V^K$, therefore 
%$$V^K= e_K V=e_K \left(Gv\right) \subseteq {\mathcal{H}_K(G)}v\subseteq V^K.$$
%This shows that $V^K$ is a simple $\mathcal{H}_K(G)$-module as desired.
%
%\end{proof}

\begin{proposition} \label{PR:GtoHecke}
Let $V$ be an irreducible smooth representation of $G$ and let $K\co G$.
Then either $V^K=0$, or $V^K$ is a simple $\mathcal{H}(G,K)$-module.
\end{proposition}

\begin{proof}
We need to show that $\mathcal{H}(G,K)\cdot v= V^K$ for all $0\neq v\in V^K$.
Let $u\in V^K$. Since $V$ is irreducible,
there are $g_1,\dots,g_n\in G$ and $\alpha_1,\dots,\alpha_n\in \C$
such that 
$\sum_i\alpha_ig_iv=u$.
Let $f=e_K\star (\sum_i \alpha_i\mu(K)^{-1}\charfunc{g_iK})\in \mathcal{H}(G,K)$.
Then
\[f\cdot v=e_K\cdot (\sum_i\alpha_i\mu(K)^{-1}\mathds{1}_{g_iK})\cdot v=e_K\cdot \sum_i\alpha_i g_iv=e_Ku=u\]
and we conclude that $u\in \mathcal{H}(G,K)\cdot v$.
\end{proof}

\begin{proposition} \label{PR:GtoHeckeUnit}
Let $V$ be an irreducible unitary representation of
$G$ and let $K\co G$. Then either $V^K=0$,
or $V^K$ is an irreducible unitary representation of $\mathcal{H}(G,K)$.
\end{proposition}

%\begin{proof}
%We proceed in the exact same fashion as in Proposition \ref{GtoHecke}. We get that for every $0\neq v\in V^K$ and every $g\in G$, there is $f\in \mathcal{H}_K(G)$ such that $fv= e_K(gv)$. Therefore $e_K (Gv) \subset \mathcal{H}_K(G)v$, passing to the closures we get 
% $$ V^K= e_KV= e_K \overline{Gv} = \overline{e_K (Gv))}\subseteq \overline{e_K \left(H_K(G)v\right)}\subseteq V^K .$$
%\end{proof}

\begin{proof}
Similarly to the proof of Proposition \ref{PR:GtoHecke}, we need
to show that $V^K$ is the closure of $\calH(G,K)\cdot v$ for
all $0\neq v\in V^K$.
Since $V$ is topologically irreducible, $\Span{gv\,|\, g\in G}$ is dense in $V$.
For all $g\in G$, we have $\mu^{-1}(K)\charfunc{gK}\cdot v=gv$ and thus,
$\Span{\charfunc{gK}\cdot v\,|\,g\in G}$ is dense in $V$.
This in turn implies that
$\Span{e_K\cdot \charfunc{gK}\cdot v\,|\,g\in G}$ is dense in $e_KV=V^K$.
Since $e_K\star \charfunc{gK}\in \calH(G,K)$, it follows
that $\calH(G,K)\cdot v$ is dense in $V$.
\end{proof}

\subsection{Miscellaneous Results}
\label{subsec:misc}

%\subsection{Amistur-Levitzki and density theorems}

We finish this   section with recalling the Amistur--Levitzky theorem
and two density theorems due to Jacobson and von Neumann. These will be needed
in proving Theorem~\ref{TH:A}.

%We state key theorems to prove theorem A. 

\medskip

To state the Amitsur--Leviski theorem, let $\Z\langle x_1,\dots,x_n\rangle$
denote the ring of polynomials in $n$ non-commuting variables
over $\Z$. The $n$-th \emph{standard polynomial} is 
define by
\[
S_n(x_1,\dots,x_n)=\sum_{\sigma }\text{sgn}(\sigma)x_{\sigma(1)}\dots x_{\sigma(n)} 
\]
where $\sigma$ runs over all permutations on $\{1,\dots,n\}$.
Recall that
a polynomial $f\in \Z\langle x_1,\dots,x_n\rangle$ is called an \emph{identity}
of a ring $R$ if $f$ vanishes on all inputs from $R$.

%\begin{definition}[Standard polynomials, Polynomial identity] Define the \textbf{standard polynomial} of degree $n$ to be $S_n(x_1,\dots,x_n)=\sum_{\sigma\in S_n}\text{sgn}(\sigma)x_{\sigma(1)}\dots x_{\sigma(n)}$. This makes sense as an element of the free (non commutative) algebra $R\left\langle x_1,\dots,x_n\right\rangle$ for every ring $R$.
%\end{definition}

\begin{theorem}[{Amitsur--Levitzki  \cite{amitsur-levitzki}}]\label{amitsur-levitzki}
Let $V$ be a $\C$-vector space and let $n\in\N$.
The polynomial $S_{n}(x_1,\dots,x_n)$ is a polynomial identity of $\End_{\C}(V)$ if and only
if $n\geq 2\dim  V$.
%The polynomial $S_{r}$ is not a polynomial identity of $\Mat_n(\C)$ if $r<2n$. 
\end{theorem}

%\begin{theorem}[Amitsur-Levitzki]\label{amitsur-levitzki} Let $R$ be a commutative ring, then $\mathrm{M}_n\p{R}$ is canceled by $S_{2n}$. In other words, for every $n\times n$ matrices $M_1,\dots,M_{2n}$ over $R$ we have $S_{2n}\p{M_1,\dots,M_{2n}}=0$.
%\end{theorem}
%\begin{proof}
%See \cite{amitsur-levitzki}.
%\end{proof}

\begin{theorem}[Jacobson density theorem]\label{jacobson_density}
	Let $R$ be a ring, let $M$ be a simple left $R$-module, and let $D$ be the
	$R$-endomorphism ring of $M$. %, which we regard as acting on $M$ on the \emph{right}.
	Give $M$ the discrete topology and $\End_D(M)$ the topology induced
	by the product topology on $M^M$ (the point-wise convergence topology).
	Then the image of $R$ in $\End_D(M)$ under the map sending $r$ to $(m\mapsto rm)$
	is dense in $\End_D(M)$.
\end{theorem} 

%\begin{theorem}[Jacobson density theorem]\label{jacobson_density}
%	Let $R$ be a ring, let $M$ be a simple left $R$-module, and let $D$ be the
%	$R$-endomorphism ring of $M$. %, which we regard as acting on $M$ on the \emph{right}.
%	Then for every $D$-endomorphism  $f:M\to M$ and a finite subset $X\subseteq M$,
%	there exists  $r\in R$ such that
%	$f(x)=rx$ for all $x\in X$.
%\end{theorem} 

%\begin{theorem}[Jacobson density theorem]\label{jacobson_density} 
%Let $R$ be a ring and let $S$ be a simple $R-$module. Write $D=\mathrm{End}_R(S)$, so $S$ can be seen as an $(R,D)-$bimodule. 	Then $R$ is dense in $\mathrm{End}_D\p{S}$ where $S$ is given the discrete topology and $\mathrm{End}_D\p{S}$ the product topology. In other words, if $\alpha\in \mathrm{End}_D\p{S}$ and $X\subset S$ is a finite $D-$linearly independent subset, then there is $r\in R$ such that $rx=x\alpha= \alpha(x)$ for all $x\in X$. 
%\end{theorem}

\begin{proof}
See, for instance, \cite[Theorem 13.14]{jacobson_density}. %, p.185 
\end{proof}

%A special case of Jacobson's density theorem is known as Burnside's theorem on matrix algebras.
%
%\begin{theorem}[Burnside's theorem]
%	Let $V$ be a finite-dimensional $\C$-vector space
%	and let $A$ be a $\C$-subalgebra of $\End_{\C}(V)$
%	such that $V$ is a simple $A$-module.
%	Then $A=\End_{\C}(V)$.
%\end{theorem}
%
%\begin{proof}
%	The endomorphism ring $\End_A(V)$ is a finite-dimensional division algebra
%	over $\C$ and hence coincides with $\C$. The theorem now follows
%	from Jacobson's density theorem by taking $R=A$, $M=V$ and $X$ to be
%	a $\C$-basis of $V$.
%\end{proof}

Given a Hilbert space $V$, let $\calB(V)$ denote
the algebra of bounded linear automorphisms of $V$.
For a subset $S\subseteq \calB(V)$,
we let $S'$ denote the commutant of $S$, namely, the set of elements
of $\calB(V)$ that commute with all elements in $S$.  

%\begin{definition}[Commutant] Let $\mathcal{H}$ be a Hilbert space and $\mathcal{B}\p{\mathcal{H}}$ its space of bounded operators. If $\mathcal{S}$ is a nonempty subset of $\mathcal{B}\p{\mathcal{H}}$ we define 
%$$\mathcal{S}' = \left\lbrace T\in \mathcal{B}\p{\mathcal{H}}: TS=ST\text{ for all }S\in \mathcal{S}\right\rbrace $$ the \textbf{commutant} of $\mathcal{S}$ and call it $S'$.
%\end{definition}

%\begin{theorem}[Von Neumann density theorem]
%\label{von_neumann_density} 
%Let $\mathcal{H}$ be a Hilbert space and let $A\subset \mathcal{B}\p{\mathcal{H}}$ be a unital subalgebra closed under taking adjoints. Then $A$ is dense in $A''$ in the strong  operator topology (and hence also in the weak operator topology). 
%\end{theorem}
%\begin{proof}
%See \cite[Theorem 9.3.3, p.888]{Palmer_algebra}.
%\end{proof}

\begin{theorem}[Von Neumann density theorem]
\label{von_neumann_density} 
Let $V$ be a Hilbert space and let $A\subset \calB(V)$ be a unital subalgebra closed under taking adjoints. 
Then $A$ is dense in $A''$ in the strong  operator topology (and hence also in the weak operator topology). 
\end{theorem}

\begin{proof}
See, for instance, \cite[Theorem 9.3.3]{Palmer_algebra}.
\end{proof}

% \section{Relation between uniform admissibility of smooth and unitary representations}
\section{Proof of Theorem~\ref{TH:A}}

%The goal of this section is to link results on uniform admissibility proved for smooth and unitary representations of groups.

%The goal of this section is to prove Theorem~\ref{TH:A}.
%In fact, we shall prove a slightly stronger version. 

%We shall prove a slightly stronger version of Theorem~\ref{TH:A}:
%  
%\begin{theorem}\label{IUA_smooth_unit}
%%Let $K\leq G$ be a compact open subgroup. Then the following are equivalent: 
%Let $G$ be an $\ell$-group and let $K$ be a compact open subgroup with $|G/K|<{2^{\aleph_0}}$. 
%Then the following conditions are equivalent:
%\begin{enumerate}
%\item[(i)] All irreducible unitary representations $V$ satisfy $\dim\p{V^K}\leq n$.
%\item[(ii)] All irreducible smooth representations $V$ satisfy $\dim\p{V^K}\leq n$.
%\end{enumerate}
%\end{theorem}
 
We begin with several results about $\C$-algebras and $*$-algebras.
The Jacobson radical of a ring $R$ is denoted $\Jac(R)$.

%\medskip
%
%The following proposition is known to experts, but we include
%a proof here for the sake of completeness and also because we were unable to find an explicit
%reference.

\begin{proposition}\label{PR:algebra}
	Let $A$ be a unital $\C$-algebra. Consider the conditions:
	\begin{enumerate}
\item[(a)] Every simple $A$-module has $\C$-dimension at most $n$.
\item[{(b)}] $A/\Jac(A)$ has $S_{2n}$ %(see~\ref{subsec:misc}) 
as a polynomial identity. 
	\end{enumerate}
	Then (a) implies (b), and the converse holds provided $\dim_\C A<{2^{\aleph_0}}$.
\end{proposition}

\begin{proof}
%	This is well-known, but we include a proof     the sake of completeness and the 
%	lack of appropriate reference.
%	
	Suppose (a) holds.
	Let $\mathscr{S}$ be a set of representative for the isomorphism
classes of simple left $A$-modules.
	By the definition of $\Jac(A)$,  the map
\[ A/\Jac(A) %\xhookrightarrow{\qquad} 
\to \prod_{M\in \mathscr{S}}\End_\C(M) 
\]
sending $a+\Jac(A)\in A/\Jac(A)$ to $(x\mapsto ax)\in\End_\C(M)$ in the $M$-component, is injective.
Condition (a) and Theorem~\ref{amitsur-levitzki} imply that $S_{2n}$
is a polynomial identity of $\End_{\C}(M)$ for all $S\in \mathscr{S}$,
and hence also of $A/\Jac(A)$. 

	Suppose now that (b) holds and $\dim_\C A<{2^{\aleph_0}}$.
	Let $M$ be a simple $A$-module.
	Consider $M$ as a discrete topological space and make
	$\End_{\C}(M)$ into a topological ring by giving it the topology 
	induced from the product topology on $M^M$.
	By Proposition~\ref{LM:Schur-smooth}, $\End_A(M)=\C$,
	and hence, by  Theorem~\ref{jacobson_density},
	the image of $A$ in $\End_\C(M)$ is dense.
	Since the map $A\to \End_{\C}(M)$ factors through $A/\Jac(A)$,
	condition (b) implies that $S_{2n}$ vanishes on all inputs from dense subset
	of $\End_{\C}(M)$. Since the topology on $\End_{\C}(M)$ is Hausdorff,
	it follows that $S_{2n}$ is a polynomial identity of $\End_{\C}(M)$.
	Thus, $\dim M\leq n$ by Theorem~\ref{amitsur-levitzki}.
\end{proof}

\begin{lemma}\label{LM:Jac-unitary}
	Let $A$ be a unitary $*$-algebra and let $V$
	be a unitary representation of $A$.
	Then $\Jac(A)\cdot V=0$.
\end{lemma}

\begin{proof}
	Let $a\in \Jac(A)$. Suppose first that $a^*=a$.
	Then the operator $T_a:V\to V$ given by $T_a(v)=av$
	is self-adjoint. Since $a\in \Jac(A)$, the element
	$\lambda-a$ is invertible in $A$ for all $0\neq \lambda \in \C$.
	As a result,  $\mathrm{Spec}(T_a)=\{0\}$, and 
	since $T$ is self-adjoint, this forces $T=0$ and $a\cdot V=0$.
	
	The general case follows by noting that any
	$a\in \Jac(A)$ satisfies
	$a=\frac{a+a^*}{2}+i(\frac{a-a^*}{2i})$ with   $\frac{a+a^*}{2}$ and
	$\frac{a-a^*}{2i} $
	being elements of $\Jac(A)$ fixed by $*$. 
\end{proof}

\begin{proposition}\label{PR:star-alg}
	Let $A$ be a unital $*$-algebra. Consider the   conditions:
	\begin{enumerate}
\item[(a$'$)] Every irreducible unitary representation of $A$  has $\C$-dimension at most $n$.
\item[{(b)}] $A/\Jac(A)$ has $S_{2n}$ %(see~\ref{subsec:misc}) 
as a polynomial identity.
	\end{enumerate} 
	Then (b) implies (a$'$), and the converse holds provided $A/\Jac(A)$ admits a faithful unitary
	representation.
\end{proposition}

\begin{proof}
%	The proof that (ii) implies (i$'$) is follows the same idea
%	as the proof that (ii) implies
%	(i) in Proposition~\ref{PR:algebra}, with some modifications.
	Suppose (b) holds and let $V$ be an irreducible unitary
	representation of $A$. %By Lemma~\ref{LM:Jac-unitary}, 
	We may
	replace $A$ with its image in $\calB(V)$  and hence assume
	that $\Jac(A)=0$, by Lemma~\ref{LM:Jac-unitary}.
	Now,    by Proposition~\ref{LM:Schur-unitary}, the commutant
	$A'$ is $\C$, and hence $A''=\calB(V)$.
	By Theorem~\ref{von_neumann_density}, $A$ is dense in $\calB(V)$
	in the strong operator topology. Condition (b) therefore implies
	that $S_{2n}$ is a polynomial identity of $\calB(V)$.
	It is easy to check that this is impossible if $V$ is infinite-dimensional,
	whereas in the finite-dimensional case, Theorem~\ref{amitsur-levitzki}
	implies that $\dim V\leq n$, as required.
	
	Suppose now that (a$'$) holds and $A/\Jac(A)$ admits a faithful unitary
	representation  $V$. 
	%Then $\Jac(A)=0$ by Lemma~\ref{LM:Jac-unitary}.
	By Lemma~\ref{LM:Jac-unitary}, we may replace $A$ with its image in 
	$\calB(V)$ and assume that $\Jac(A)=0$ hereafter.
	Let $0\neq a\in A$. The the operator 
	$v\mapsto av:V\to V$ is nonzero, and hence, by
	\cite[Corollary 2.30]{Uriya}, there exists an irreducible
	unitary representation $V_a$ of $A$ such that the image of
	$a$ in $\calB(V_a)$ is nonzero.
	Consequently, the ring homomorphism
	\[
	A  \longhookrightarrow \prod_{a\neq A}\mathcal{B}(V_a)  
	\]
	is injective. By condition (a$'$) and Theorem~\ref{amitsur-levitzki}, 
	$S_{2n}$ is a polynomial identity of each of the factors $\calB(V_a)$,
	and \emph{a fortiori} of $A$ as well.
\end{proof}

We finally prove Theorem~\ref{TH:A} by establishing the following
slightly stronger version.
%In fact, we shall prove a slight stronger version.
Notice that if   $|G/K|<{2^{\aleph_0}}$ holds
for some $K\co G$, then it holds for all $K\co G$.

\begin{theorem}\label{TH:IUA_smooth_unit}
%Let $K\leq G$ be a compact open subgroup. Then the following are equivalent: 
Let $G$ be an $\ell$-group, let $K$ be a compact open subgroup with $|G/K|<{2^{\aleph_0}}$,
and let $n\in\N$. 
Then the following conditions are equivalent:
\begin{enumerate}
\item[(a)] All irreducible unitary representations $V$ of $G$ satisfy $\dim\p{V^K}\leq n$.
\item[(b)] All irreducible smooth representations $V$ of $G$ satisfy $\dim\p{V^K}\leq n$.
\end{enumerate}
In fact, the implication (b)$\Longrightarrow$(a) holds without assuming $|G/K|<{2^{\aleph_0}}$.
\end{theorem}

\begin{proof} %[Proof of Theorem \ref{IUA_smooth_unit}]
Let $A=\calH(G,K)$.
By Propositions~\ref{PR:GtoHecke} and~\ref{PR:GtoHeckeUnit},
it is enough to show that every simple $A$-module is of dimension $\leq n$
if and only if every irreducible unitary representation of $A$ is of dimension $\leq n$.
This would follow from Propositions~\ref{PR:algebra} and~\ref{PR:star-alg}
if we check that $\dim_{\C}A<{2^{\aleph_0}}$ and $A$ admits a faithful unitary
representation.
The former condition holds since $|G/K|<{2^{\aleph_0}}$,
whereas for the second condition one can consider the left action
of $A=\calH(G,K)$ on $\mathrm{L}^2(K{\setminus}G)$ by convolution.
Indeed, if $a\star f=0$ for all $f\in \mathrm{L}^2(K{\setminus}G)$,
then $a=a\star e_K=0$.
%%%%%% Uriya: L^2(K\G) is indeed a unitary representation.
%
%The proof boils down to proving that every smooth $\mathcal{H}_K(G)$ modules have dimension at most $n$ if and only if every unitary representation of $\mathcal{H}_K\p{G}$ has dimension at most $n$. 
%
%
%
%$(\Rightarrow)$ Suppose that every irreducible representation of $\mathcal{H}_K(G)$ has dimension at most $n$. Apply Theorem \ref{irred_unit_equivs} with $A=\mathcal{H}_K(G)$ to get that every irreducible unitary representation of $\mathcal{H}_K(G)$ has dimension at most $n$. 
%
%
%$(\Leftarrow)$ Suppose that every representation of $\mathcal{H}_K(G)$ has dimension at most $n$. We want to apply \ref{irred_unit_equivs} again.
%Consider the left action of $\mathcal{H}_K(G)$ on $L^2\p{G}$, the square integrable functions on $G$, given by $fg=f\star g$ for all $f\in \mathcal{H}_K(G)$ and $g\in L^2(G)$. It is straightforward to see that it defines as a unitary representation. Let us prove that this representation if faithful. Note that $e_K\in L^2\p{G}$ and for all $f\in \mathcal{H}_K(G)$ we have $f\star e_K= f$. Suppose $f,g\in \mathcal{H}_K(G)$ give the same morphism. Then $$f=f\star e_K = g \star e_K = g,$$ as desired.
%
%We conclude that $\mathcal{H}_K(G)/J\p{\mathcal{H}_K(G)}$ has a faithful unitary representation, so by Theorem \ref{irred_unit_equivs} all irreducible representations of $\mathcal{H}_K(G)$ have dimension at most $n$.
\end{proof}

\begin{rem}\label{RM:counterexamples}
	(i) The assumption that   $|G/K|<{2^{\aleph_0}}$ in Theorems~\ref{TH:A}
	and~\ref{TH:IUA_smooth_unit} is necessary in general.
	For instance, consider the function field
	$\C(t)$ and let $G$ be the multiplicative group $\C(t)^\times$ with the discrete
	topology.
	Since $G$ is commutative and locally compact, 
	%By virtue of \cite[Cor.~2.7, Prp.~6.5]{Uriya}, 
	any unitary
	representation of $G$ is $1$-dimensional.
	However, $\C(t)$ with the $G$-action induced by the product
	in $\C(t)$ is an infinite-dimensional
	irreducible smooth representation of $G$.
	
	(ii) The assumption that $\dim_\C A<{2^{\aleph_0}}$
	in Proposition~\ref{PR:algebra} cannot be dropped
	in general. 
	For example, take $A=\C(t)$ and note that
	$S_2(x_1,x_2)=x_1x_2-x_2x_1$ is a polynomial identity
	of $A$, while $A$   itself is a simple
	$A$-module of infinite $\C$-dimension. 
	
	(iii) The assumption that $A/\Jac(A)$ admits a faithful unitary
	representation in Proposition~\ref{PR:star-alg} is necessary
	in general.
	A counterexample can be constructed as follows: Let $*$ denote
	the unique involution of $\C(t)$ extending the complex
	conjugation and fixing $t$.
	Let $A=\Mat_2(\C(t))$
	and let $*:A\to A$ be the involution given by
	$(f_{ij})\mapsto (f^*_{ji})$.
	Then $S_2(x_1,x_2)=x_1x_2-x_2x_1$  is not a polynomial
	identity of $A$, but all irreducible unitary
	representations of $A$ have dimension~$\leq 1$.
	In fact, the latter   holds vacuously, since $A$ has no
	nonzero unitary representations. To see this, note that
	if $V$ is a Hilbert space and $A\to \calB(V)$ is a $*$-homomorphism,
	then the image of $[\begin{smallmatrix} t & 0 \\ 0 & t \end{smallmatrix}]$
	in $\calB(V)$ is a self-adjoint operator with an empty spectrum, forcing $V=0$.
\end{rem}

 \section{Proof of Theorem~\ref{TH:B}}

%In this section we prove the following theorem, which particularly
%implies Theorem~\ref{TH:B} from the introduction.

	Theorem~\ref{TH:B}  follows readily from the following theorem.

\begin{theorem}
	Let $H$ be an open  subgroup of $G$ of finite index  
	and let $K$ be a compact open subgroup of $H$. Then there is $n\in \N$ 
	and a compact open subgroup $L\leq K$ such
	that:
	\begin{enumerate}
		\item[(i)] For any irreducible smooth
		representation $V$ of $H$, there are   irreducible 
		smooth representations $U_1,\dots,U_m$ ($m\leq n$) of $G$
		such that $\dim V^K\leq   \sum_i\dim U_i^L$.
		\item[(ii)] For any irreducible smooth
		representation $U$ of $G$, there are   irreducible 
		smooth representations $V_1,\dots,V_m$ ($m\leq n$) of $H$
		such that $\dim U^K\leq   \sum_i\dim V_i^L$.
	\end{enumerate}
\end{theorem}
 
% We will say that a totally disconnected locally compact group has property \textbf{IA} (resp. \textbf{IUA}) if every irreducible smooth representation is admissible (resp. uniformly admissible).
%
%
%
% Our goal here is to prove the following : 
% 
% \begin{theorem}
% Let $H$ be a closed finite index subgroup of $G$, then $G$ has \textbf{IA} (resp. \textbf{IUA}) if and only if $H$ has \textbf{IA} (resp. \textbf{IUA}).
% \end{theorem}

\begin{proof}
	Since $[G:H]$ is finite, $H$ has only finitely
	many conjugates. Their intersection, $N$, is an open  normal subgroup
	of finite index. We may replace $K$ with $K\cap N$.
	The theorem would now follow if we establish it for
	the  pairs $(G,N)$ and $(H,N)$. It is therefore enough to prove
	the theorem when $H$ is normal in $G$.
	Under this extra assumption, we  shall see that $L=K$  
	and $n:=[G:H]$ will suffice.

\smallskip

	(i) Let $g_1,\dots,g_n$ be representatives
	for the $H$-cosets in $G$.
	Let $V_i$ be the irreducible smooth representation of $H$ obtained by giving
	$V$  the $H$-action $(h,v)\mapsto (g_ihg_i^{-1})v$ (note that
	$g_ihg_i^{-1}\in H$ since $H\lhd G$). 
	
	Consider the  (non-normalized) induction 
	\[\mathrm{Ind}_H^G(V) := \{f:G\to V\,|\, \text{$f(hg)=h\cdot f(g)$ for all $h\in H$, $g\in G$}\}\ . \]
	As usual, the group $G$ acts on $\mathrm{Ind}_H^G(V)$ by translations on the right.
	Let $V'_i$ denote the functions in $\mathrm{Ind}_H^G(V)$
	supported on $g_iH=Hg_i$. It is easy to see that $V'_i$
	is an $H$-subrepresention of $\mathrm{Ind}_H^G(V)$,
	that $\mathrm{Ind}_H^G(V)=\bigoplus_{i=1}^n V'_i$
	and that $f\mapsto f(g_i):V'_i\to V_i$
	defines an isomorphism of $H$-modules.
	Thus, the length of $\mathrm{Ind}_H^G(V)$ as a complex
	$H$-module is   $n$, and so the length of $\mathrm{Ind}_H^G(V)$ as a complex
	$G$-module is at most $n$. 
	
	Let $U_1,\dots,U_m$ denote the composition
	factors of $\mathrm{Ind}_H^G(V)$, regarded as a $G$-module.
	Then $m\leq n$, and $\sum_i\dim U_i^K=\dim\mathrm{Ind}_H^G(V)^K	\geq \dim V^K$,
	so we have proved (i).

\smallskip

	(ii) Let $U$ be an irreducible smooth representation of $G$, and let
	$g_1,\dots,g_n$ be as in (i). Since $U$ is a simple   $G$-module,
	it is finitely generated. As $[G:H]<\infty$, it follows
	that $U$ is also finitely generated as a complex $H$-module.
	Thus, there exists a maximal complex $H$-submodule $M_0\leq U$.
	Since $H$ is normal in $G$, the space $gM_0$ is  a maximal $H$-submodule for
	all $g\in G$. Let $\calM=\{gM_0\,|\,g\in G\}$. Then $\calM$ consists of 
	at most $[G:H]=n$ elements, call them $M_1,\dots,M_m$.
	Write $V_i=U/M_i$ and set $V=\bigoplus_{i=1}^m U/M_i=\bigoplus_{M\in \calM} U/M$.
	Then action of $H$ on $V$ extends to a $G$-action
	by setting $g\cdot \bigoplus_{M\in\calM}(v_M+M)=\bigoplus_{M\in\calM}(gv_{g^{-1}M}+M)$
	and the map $u\mapsto \bigoplus_{M\in\calM} (u+M):U\to V$ defines
	a nontrivial $G$-homomorphism, which must be injective since $V$ is irreducible.
	It follows that $\dim U^K\leq\dim V^K=\sum_{i=1}^m\dim V_i^K$, 
	and the proof is complete.
%	
%
%$(\Leftarrow)$ : Now assume that $H$ has \textbf{IA}. Let $V$ be an irreducible smooth representation of $G$. If $0\neq v\in V$ then $V=\mathrm{span}(Gv)=\sum_{i=1}^n\mathrm{span}(Hg_iv)$. We get that $V$ is a finitely generated $H-$module, therefore it has an irreducible nontrivial quotient $W$. Let $M\leq V$ be the maximal proper submodule of $V$ such that $W\cong V/M$.  Since $H$ is normal, for all $g\in G$ the space $gM$ is an $H-$module. Note that $G$ acts on $\{g_1,\dots,g_n\}$ by $gg_i=g_j$ where $gg_iH=g_jH$. This action lets us consider  $\bigoplus_{i=1}^nV/g_iM$ as a $G$-module where $G$ permutes the factors. Consider the $G$-module morphism  
%$ \psi : V\rightarrow \bigoplus_{i=1}^nV/g_iM$  defined by $\psi(v)=\bigoplus_{i=1}^nv+g_iM$.
%
%This map is injective by irreducibility of $V$. View it as an $H-$module morphism, then $V$ embeds in a semisimple $H-$module of length $n$ hence it is admissible as an $H$-module, and so it is admissible as a representation of $G$ as well (we can take compact open subgroups in $H$ without loss of generality).
%
%
%Now if $H$ has \textbf{IUA}, take $K\leq G$ compact open, and assume $K\leq H$. We have 
%$$ \dim{V^K}\leq \dim\left(\left(\bigoplus_{i=1}^nV/g_iM \right)^K\right)\leq \sum_{i=1}^n\dim\left((V/g_iM)^K\right)\leq nN_H(K),$$
%hence $G$ has \textbf{IUA}.
%
\end{proof} 

\begin{rem}
	The proof of part (ii) also shows that any irreducible complex $G$-module
	has   finite length as a complex $H$-module.
\end{rem}

\bibliographystyle{plain}
\bibliography{biblio}

\end{document}